\documentclass{elsarticle} 

\usepackage{amssymb,amsmath,amsthm}
\usepackage{hyperref}
\usepackage[capitalize]{cleveref}

\newcommand{\MR}[1]{\href{http://www.ams.org/mathscinet-getitem?mr=#1}{MR#1}}

\numberwithin{equation}{section}

\theoremstyle{plain}
\newtheorem{thm}[equation]{Theorem}
\newtheorem{prop}[equation]{Proposition}
\newtheorem{cor}[equation]{Corollary}
\newtheorem{lem}[equation]{Lemma}

\crefname{thm}{Theorem}{Theorems}
\Crefname{thm}{Theorem}{Theorems}
\crefname{cor}{Corollary}{Corollaries}
\Crefname{cor}{Corollary}{Corollaries}
\crefname{lem}{Lemma}{Lemmas}
\Crefname{lem}{Lemma}{Lemmas}

\theoremstyle{definition}
\newtheorem{notation}[equation]{Notation}
\newtheorem{defn}[equation]{Definition}

\newtheorem*{ack}{Acknowledgments}

\theoremstyle{definition}
\newtheorem{rem}[equation]{Remark}

\crefformat{rems}{Remark~#2#1#3}

 \newcounter{case}
 \newenvironment{case}[1][\unskip]{\refstepcounter{case}
 \medskip \noindent \textbf{Case \thecase.}\em\ #1\ }{\unskip\upshape}
 \renewcommand{\thecase}{\arabic{case}}
 
\crefformat{case}{Case~#2#1#3}
\Crefname{case}{Case}{Cases}

\newenvironment{claim}[1][\unskip]{\medskip \noindent 
{\em Claim.\ #1}}{\medbreak}
 
 \renewcommand{\Cref}{\cref}
 \newcommand{\pref}[1]{{\upshape(\ref{#1})}}
 \newcommand{\fullref}[2]{\ref{#1}\pref{#1-#2}}
 \newcommand{\fullcref}[2]{\cref{#1}\pref{#1-#2}}

\newcommand{\rational}{\mathbb{Q}}
\newcommand{\integer}{\mathbb{Z}}
\newcommand{\real}{\mathbb{R}}
\newcommand{\complex}{\mathbb{C}}
\newcommand{\ints}{\mathcal{O}}
\newcommand{\Rrank}{\rank_{\real}}
\newcommand{\Qrank}{\rank_{\rational}}
\newcommand{\Res}[1]{\mathop{\mathrm{Res}}\nolimits_{#1}}

\newcommand{\one}{_S^{(1)}}

\DeclareMathOperator{\rank}{rank}
\DeclareMathOperator{\GL}{GL}
\DeclareMathOperator{\graph}{graph}
\DeclareMathOperator{\Rad}{Rad}
\DeclareMathOperator{\unip}{unip}
\DeclareMathOperator{\Char}{char}

\newcommand{\Zar}[1]{\overline{#1}{}}
\newcommand{\iso}{\cong}
\newcommand{\bigtimes}{\mathop{\hbox{\Large$\times$}}}

\renewcommand{\AA}{\mathbf{A}}
\newcommand{\BB}{\mathbf{B}}
\newcommand{\CC}{\mathbf{C}}
\newcommand{\GG}{\mathbf{G}}
\newcommand{\HH}{\mathbf{H}}

\newcommand{\MM}{\mathbf{M}}
\newcommand{\NN}{\mathbf{N}}
\newcommand{\RR}{\mathbf{R}}
\newcommand{\TT}{\mathbf{T}}
\newcommand{\UU}{\mathbf{U}}
\newcommand{\VV}{\mathbf{V}}
\newcommand{\WW}{\mathbf{W}}

\newcommand{\ZZ}{\mathbf{Z}}
\DeclareMathOperator{\GGL}{\bf GL}

\newcommand{\almsubset}{\stackrel{.}{\subseteq}}

\newcommand\bigset[2]{\left\{\, #1 
 \mathrel{\left| \vphantom {\left\{ #1 \mid #2 \right\} }
 \right.} #2 \,\right\} }

\makeatletter
\newcommand{\noprelistbreak}{\unskip\medskip\@nobreaktrue\nopagebreak} 
\makeatother

\makeatletter
\def\cref@thmoptarg[#1]#2#3#4{%
  \ifhmode\unskip\unskip\par\fi%
  \normalfont%
  \trivlist%
  \let\thmheadnl\relax%
  \let\thm@swap\@gobble%
  \thm@notefont{\fontseries\mddefault\upshape}%
  \thm@headpunct{.}
  \thm@headsep 5\p@ plus\p@ minus\p@\relax%
  \thm@space@setup%
  #2
  \@topsep \thm@preskip               
  \@topsepadd \thm@postskip           
  \def\@tempa{#3}\ifx\@empty\@tempa%
    \def\@tempa{\@oparg{\@begintheorem{#4}{}}[]}%
  \else%
    \refstepcounter[#1]{#3}
    \def\dth@counter{} 
    \def\@tempa{\@oparg{\@begintheorem{#4}{\csname the#3\endcsname}}[]}%
  \fi%
  \@tempa}
\makeatother

\begin{document}

\title{Nonarchimedean superrigidity of solvable $S$-arithmetic groups}

\author[uleth]{Dave Witte Morris\corref{cor1}}
\ead{Dave.Morris@uleth.ca}
\ead[url]{http://people.uleth.ca/~dave.morris/}
\address[uleth]{Department of Mathematics and Computer Science, 
	\\ University of Lethbridge, 
	Lethbridge, Alberta, T1K\,6R4, Canada
	}

\author[uchgo]{Daniel Studenmund}
\ead{dhs@math.uchicago.edu}
\ead[url]{http://math.uchicago.edu/~dhs/}
\address[uchgo]{Department of Mathematics, University of Chicago, Chicago, IL 60637}
    
\cortext[cor1]{Corresponding author}

\begin{abstract}
Let $\Gamma$~be an $S$-arithmetic subgroup of a solvable algebraic group~$\GG$ over an algebraic number field~$F$, such that the finite set~$S$ contains at least one place that is nonarchimedean. We construct a certain group $H = \bigl(\Zar{\Res{F/\rational} \Gamma}\bigr){}_{\Char S}^\Gamma$, such that if $L$~is any local field and $\alpha \colon \Gamma \to \GL_n(L)$ is any homomorphism, then $\alpha$ virtually extends (modulo a bounded error) to a continuous homomorphism defined on some finite-index subgroup of~$H$. 
In the special case where $F = \rational$, $\Rrank \GG = 0$, and $\Gamma$~is Zariski-dense in~$\GG$, we may let $H = \GG_S$.

We also point out a generalization that does not require $\GG$ to be solvable.
\end{abstract}

\begin{keyword}
superrigidity, solvable group, algebraic group, $S$-arithmetic group, representation
\MSC[2010]  20G30 \sep 22E40
\end{keyword}

\maketitle

\section{Introduction}

Roughly speaking, a subgroup~$\Lambda$ of a topological group~$G$ is said to be ``superrigid'' if every finite-dimensional representation of~$\Lambda$ is the restriction of a continuous representation of~$G$. However, this need only be true up to finite-index subgroups and modulo a compact subgroup of the range:

\begin{defn}[{cf.~\cite[Thm.~2, p.~2]{MargBook}}] \label{SuperDefn}
 Let $\Gamma$ be a (countable) subgroup of a topological group~$G$, and let $L$ be a local field. 
We say that $\Gamma$ is \emph{$L$-superrigid} in~$G$ if, for every
 homomorphism $\alpha \colon \Gamma \to \GL_n(L)$,
 there are
 \noprelistbreak
 \begin{itemize}
 \item a finite-index, open subgroup~$G_0$ of~$G$,
 \item a finite-index subgroup~$\Gamma_0$ of $\Gamma \cap G_0$,
 \item a finite, normal subgroup~$N$ of $H$, where $H$ is the group of $L$-points of the Zariski closure of $\alpha(\Gamma_0)$,
 \item a continuous homomorphism $\widehat\alpha \colon G_0 \to H/N$,
 and
 \item a compact subgroup~$K$ of $H/N$ that centralizes $\widehat\alpha(G_0)$,
 \end{itemize}
 such that $\alpha(\gamma) N \in \widehat\alpha(\gamma) K$ for all $\gamma \in \Gamma_0$.
  \end{defn}

Let $\Gamma$ be an $S$-arithmetic subgroup of a solvable algebraic group~$\GG$ over an algebraic number field~$F$. It is known that $\Gamma$ is $\real$-superrigid in a certain closed subgroup $\overline{\GG(\ints_S)}$ of $\GG_S$ \cite{Witte-NonarchSuper}. (Hence, $\Gamma$ is also $\complex$-superrigid in $\overline{\GG(\ints_S)}$.)
We now show that if $S$ contains at least one nonarchimedean place, then $\Gamma$~is also superrigid over the other local fields. (That is, it is $L$-superrigid when $L$~is a $p$-adic field~$\rational_p$ or a function field $\mathbb{F}\!_q((T))$.) However, $\overline{\GG(\ints_S)}$ will be replaced with a somewhat different group $\bigl(\Zar{\Res{F/\rational} \Gamma}\bigr){}_{\Char S}^\Gamma$.
Before stating a sample result, we recall some fairly standard terminology:

\begin{defn}
Suppose $\GG$ is an algebraic group defined over a finite extension~$F$ of~$\rational$, and $S$ is a finite set of places of~$F$. 
(All algebraic groups in this paper are assumed to be affine.)
	\begin{enumerate}
	\item Subgroups $\Gamma$ and~$\Lambda$ of~$\GG$ are \emph{commensurable} if $\Gamma \cap \Lambda$ has finite index in both~$\Gamma$ and~$\Lambda$.
	\item For any field~$L$ that contains~$F$, $\rank_L \GG$ is the dimension of any maximal $L$-split torus in~$\GG$.
	\item $\ints_S$ is the ring of $S$-integers of~$F$.
	\item A subgroup~$\Gamma$ of~$\GG$ is \emph{$S$-arithmetic} if it is commensurable to $\GG(\ints_S)$. 
	\item $\displaystyle \GG_S  = \bigtimes_{p \in S \cup S_\infty} \GG(F_v) $, where $F_v$ is the completion of~$F$ at the place~$v$, and $S_\infty$~is the set of all archimedean places of~$F$.
	\end{enumerate}
\end{defn}

\begin{rem}
In the definition of an $S$-arithmetic group, it is usually assumed that $S$ contains all of the archimedean places, but we do not make this requirement. Thus, the groups we call ``$S$-arithmetic'' would be called ``$(S \cup S_\infty)$-arithmetic'' in the usual terminology.
\end{rem}

Here is an archimedean superrigidity theorem that is easy to state:

\begin{thm}[Witte {\cite[Thm.~1.6]{Witte-NonarchSuper}}] \label{ArchSuperQrank0}
Suppose\/ $\Gamma$~is a Zariski-dense, $S$-arithmetic subgroup of a solvable algebraic\/ $\rational$-group\/~$\GG$. If\/ $\Qrank \GG = 0$, then\/ $\Gamma$ is\/ $\real$-superrigid in\/~$\GG_S$.
\end{thm}

Our results imply that if the rank of~$\GG$ is~$0$ over~$\real$, not just over~$\rational$, then $\Gamma$ is usually also superrigid over all other local fields:

\begin{thm} \label{SuperRrank0}
Suppose\/ $\Gamma$~is a Zariski-dense, $S$-arithmetic subgroup of a solvable algebraic\/ $\rational$-group\/~$\GG$. If\/ $\Rrank \GG = 0$ and $S$~contains at least one nonarchimedean place, then\/ $\Gamma$ is $L$-superrigid in $\GG_S$, for every local field~$L$.
\end{thm}

To state a more general theorem that applies to solvable algebraic groups over any algebraic number field, not only~$\rational$, we introduce additional notation.

\begin{notation} \label{GSDefn}
Let $\GG$~be a solvable algebraic group over a finite extension~$F$ of~$\rational$, and let $\Gamma$~be an $S$-arithmetic subgroup of~$\GG$, for some finite set $S$ of  places of~$F$.
	\begin{enumerate}

	\item $\Res{F/\rational} \GG$ is the $\rational$-group obtained from~$\GG$ by restriction of scalars.

	\item $\Zar{\Res{F/\rational} \Gamma}$ is the Zariski closure of~$\Gamma$ in $\Res{F/\rational} \GG$.

	\item $\Char v$ denotes the residue characteristic of a nonarchimedean place~$v$ of~$F$  (so $F_v$ is a finite extension of~$\rational_p$ if $p = \Char v$). 

	\item $\Char S = \{\, \Char v \mid \text{$v \in S$ and $v$~is nonarchimedean} \,\} $.
	
	\item  \label{GSDefn-C}
	If $\UU = \unip \GG$ is the unipotent radical of~$\GG$, and $C/\UU_S$ is the (unique) maximal compact subgroup of the abelian group $(\GG^\circ)_S/\UU_S$, where $\GG^\circ$ is the identity component of~$\GG$, then
	$\GG_S^\Gamma = \Gamma \cdot C$.
	\end{enumerate}
\end{notation}

\begin{thm} \label{SuperRes}
Suppose 
\noprelistbreak
	\begin{itemize}
	\item $\GG$ is a solvable algebraic group over a finite extension~$F$ of\/~$\rational$,
	\item $S$ is a finite set of places of~$F$,
	and
	\item $\Gamma$ is a Zariski-dense, $S$-arithmetic subgroup of\/~$\GG$.
	\end{itemize}
If either $S$ contains at least one nonarchimedean place or\/ $[\GG^\circ, \GG^\circ] = \unip \GG$, then\/ $\Gamma$ is $L$-superrigid in\/ $\bigl(\Zar{\Res{F/\rational} \Gamma}\bigr){}_{\Char S}^\Gamma$, for every local field~$L$.
\end{thm}

\begin{rem}
If $F = \rational$ (so $\Res{F/\rational} \GG = \GG$) and $\Rrank \GG = 0$, then the subgroup~$C$ in \fullcref{GSDefn}{C} is open, and $\GG_S/\Gamma$ is compact \cite[Thm.~5.5(1), p.~260]{PlatonovRapinchukBook}. This implies that $\GG_S^\Gamma = \Gamma \cdot C$ is a finite-index subgroup of $\GG_S$. Therefore, \cref{SuperRrank0} is a consequence of \cref{SuperRes}.
\end{rem}

We state and prove our main theorem in \cref{GeneralSect}.
By specializing this general result to solvable groups, \cref{SolvSect} obtains \cref{SuperRes} and other similar superrigidity theorems.
A consequence for non-solvable groups is stated in \cref{NonsolvSect}.

\begin{rem}
Our results assume that $\GG$ is defined over a field of characteristic zero, so we have nothing to say about $S$-arithmetic subgroups of solvable groups that are defined over a global field of positive characteristic. That seems to be a much more difficult problem, and we merely point out that the paper \cite{LifschitzWitte} proves a rigidity result (but not superrigidity) in a very special case.
\end{rem}

\section{A nonarchimedean superrigidity theorem} \label{GeneralSect}

We now state and prove our main result. 
Later sections of the paper explain that superrigidity results for various $S$-arithmetic groups are special cases of this theorem.

\begin{notation}
We use $\Zar{X}$ to denote the Zariski closure of a matrix group~$X$. We emphasize that this is the \textbf{Zariski} closure, not the closure in the ordinary topology.
\end{notation}

\begin{defn}
Let $L$ be a local field. A subgroup~$\Gamma$ of a topological group~$G$ is \emph{semisimply $L$-superrigid} in~$G$ if $\Gamma$ is $L$-superrigid in~$G$, and, for every homomorphism $\alpha \colon \Gamma \to \GL_n(L)$:
\noprelistbreak
	\begin{enumerate}
	\item if $\Char L = 0$, then $\Zar{\alpha(\Gamma)}$ is semisimple,
	and
	\item if $\Char L \neq 0$, then $\alpha(\Gamma)$ is contained in a compact subgroup of $\GL_n(L)$.
	\end{enumerate}
\end{defn}

\begin{thm} \label{GenSuper}
 Let
 \noprelistbreak
 	\begin{itemize}
	\item $\GG$ be a connected algebraic group over\/~$\rational$, 
	\item $S$ be a finite set of prime numbers,
	\item $\Gamma$ be a Zariski-dense subgroup of\/~$\GG$,
	and
	\item $L$ be a nonarchimedean local field.
	\end{itemize}
Write 
 \noprelistbreak
	\begin{itemize}
	\item $\GG = \MM  \TT \UU$, where
		\begin{itemize}
		\item $\MM$ is a semisimple\/ $\rational$-group,
		\item $\TT$ is a $\rational$-torus that centralizes\/~$\MM$,
		and
		\item $\UU = \unip \GG$,
		\end{itemize}
	\item $M_\Gamma = \Gamma \cap \MM$, $T_\Gamma = \Gamma \cap \TT$, and $U_\Gamma = \Gamma \cap \UU$, 
	and
	\item $\GG_S^\Gamma = T_\Gamma \, \MM_S K_T \UU_S$, where $K_T$ is the\/ \textup(unique\textup) maximal compact subgroup of\/ $\TT\!_S$.
	\end{itemize}
Assume:
 \noprelistbreak
	\begin{enumerate}
	\item \label{GenSuper-GZ} 
	$\GG(\integer) \almsubset \Gamma \almsubset \GG(\integer_S)$, where $\almsubset$ means ``has a finite-index subgroup that is contained in,''
	\item \label{GenSuper-MTU}
	$M_\Gamma T_\Gamma U_\Gamma$ has finite index in~$\Gamma$,
	\item \label{GenSuper-Msuper}
	$M_\Gamma$ is semisimply\/ $L$-superrigid in $\MM_S$,
	and
	\item \label{GenSuper-UnoZ}
	for every finite-index subgroup\/~$\Gamma'$ of\/~$\Gamma$, the group $(\Gamma' \cap \UU)/ \bigl( [\Gamma',\Gamma'] \cap \UU \bigr)$ has no infinite, cyclic quotient.
	\end{enumerate}
Then $\Gamma$ is $L$-superrigid in $\GG_S^\Gamma$.
 \end{thm}

Before proving \cref{GenSuper}, we record a few observations, mostly about unipotent groups.
First of all, note that every subgroup of a unipotent group is nilpotent, and therefore has a well-defined (Hirsch) rank, which is the supremum of the ranks of its finitely generated subgroups \cite[Defn.~2.9, p.~32]{RaghunathanBook}.

\begin{lem}[cf.\ {\cite[Thm.~2.10, p.~32]{RaghunathanBook}}]  \label{rank>dim}
If\/ $\Gamma$ is any subgroup of a unipotent algebraic group\/~$\UU$ \textup(over a field of characteristic zero\textup), then
$\rank \Gamma \ge \dim \Zar\Gamma$.
\end{lem}

\begin{proof}
Let $\pi \colon \UU \to \UU/[\UU,\UU]$ be the natural homomorphism. There is no harm in assuming $\UU = \Zar\Gamma$. Then $\Gamma \cap [\UU,\UU]$ is Zariski-dense in $[\UU,\UU]$, so, by induction on $\dim \UU$, we know that 
	$\rank \bigl( \Gamma \cap [\UU,\UU] \bigr) \ge \dim [\UU,\UU]$.
Also, $\pi(\Gamma)$ is Zariski-dense in the abelian unipotent group $\pi(\UU)$, so it is obvious that $\rank \pi(\Gamma) \ge \dim \pi(\UU)$. Therefore
	\begin{align*} \rank \Gamma 
		&= \rank \bigl(\Gamma \cap \ker\pi \bigr) + \rank \pi(\Gamma)
		=  \rank \bigl(\Gamma \cap [\UU,\UU]  \bigr) + \rank \pi(\Gamma)
		\\& \ge \dim  [\UU,\UU] + \dim \bigl( \UU/[\UU,\UU] \bigr)
		= \dim \UU 
		= \dim \Zar\Gamma 
		. \qedhere \end{align*}
\end{proof}

\begin{cor}[cf.\ {\cite[Thm.~2.11, p.~33]{RaghunathanBook}}] \label{SuperRank}
If 
	\begin{itemize}
	\item $\UU$ is a unipotent algebraic group over~$\rational_p$, for some prime~$p$,
	\item $\Gamma$ is a subgroup of\/ $\UU(\rational_p)$,
	\item $\rank \Gamma = \dim \Zar\Gamma$, 
	and
	\item $\alpha \colon \Gamma \to \GL_n(\rational_p)$ is a finite-dimensional representation of\/~$\Gamma$, such that $\alpha(\Gamma)$ is unipotent, 
	\end{itemize}
then $\alpha$ extends uniquely to a rational representation $\widehat\alpha \colon \Zar\Gamma \to  \GGL_n$. \textup(Furthermore, $\widehat\alpha$ is defined over~$\rational_p$.\textup)
\end{cor}

\begin{proof}
The uniqueness of the extension~$\widehat\alpha$ is immediate from the Zariski density of~$\Gamma$ in~$\Zar{\Gamma}$. Also, $\widehat\alpha$ must be defined over~$\rational_p$, since it maps the Zariski-dense set~$\Gamma$ of $\rational_p$-points of $\Zar{\Gamma}$ into the $\rational_p$-points of $\GGL_n$.  Therefore, we need only prove the existence of~$\widehat\alpha$.

There is no harm in assuming $\UU = \Zar\Gamma$. Let 
	$$ \graph(\alpha) = \bigl\{\, \bigl( \gamma, \alpha(\gamma) \bigr) \mid \gamma \in \Gamma \,\bigr\} \subseteq \UU \times \Zar{\alpha(\Gamma)} ,$$
and
	$$ \text{$\pi \colon \Zar{\graph(\alpha)} \to \UU$ be the composition $\Zar{\graph(\alpha)} \hookrightarrow \UU \times \Zar{\alpha(\Gamma)} \to \UU$.} $$
Then 
	$\pi \bigl( \graph(\alpha) \bigr) = \Gamma$
is Zariski-dense in~$\UU$, so $\pi \raise2pt\hbox{$\Bigl($} \Zar{\graph(\alpha)} \raise2pt\hbox{$\Bigr)$} = \UU$. 
On the other hand, we have $\graph(\alpha) \iso \Gamma$, so, from \cref{rank>dim}, we know 
	$$\dim \Zar{\graph(\alpha)} \le \rank \graph(\alpha) = \rank\Gamma = \dim \Zar\Gamma = \dim \UU .$$
Therefore $\dim \ker\pi = 0$.
Since the unipotent group $\Zar{\graph(\alpha)}$ has no nontrivial subgroups that are $0$-dimensional (in other words, finite), this implies that $\pi$ is an isomorphism of algebraic groups. Therefore, $\Zar{\graph(\alpha)}$ is the graph of a rational homomorphism $\widehat\alpha \colon \UU \to \Zar{\alpha(\Gamma)}$. Namely, $\widehat\alpha$ is the composition
	\begin{align*}
	\UU \stackrel{\pi^{-1}}{\longrightarrow} \Zar{\graph(\alpha)} \hookrightarrow \UU \times \Zar{\alpha(\Gamma)} \longrightarrow \Zar{\alpha(\Gamma)}
	. &\qedhere \end{align*}
\end{proof}

\begin{lem} \label{CpctSubsGenCpct}
Let 
\noprelistbreak
	\begin{itemize}
	\item $p$ be a prime number,
	\item $\GG$ be a connected algebraic group defined over\/~$\rational_p$,
	\item $\GG = \MM \RR$, where\/ $\MM$ is a semisimple group defined over\/~$\rational_p$, and\/ $\RR$~is the solvable radical,
	\item $K$ be a compact subgroup of\/ $\MM(\rational_p)$,
	and
	\item $K_1,\ldots,K_n$ be compact subgroups of\/ $\RR(\rational_p)$.
	\end{itemize}
Then the $p$-adic closure of\/ $\langle K, K_1,\ldots,K_n \rangle$ is compact. 
\end{lem}

\begin{proof}
Write $\RR = \TT \ltimes \UU$, where $\TT$ is a torus that centralizes~$\MM$, and $\UU$ is unipotent, and let $C$ be the unique maximal compact subgroup of the abelian group $\TT(\rational_p)$. For each~$i$, let 
	$$U_i = (KC K_i CK) \cap \UU(\rational_p) ,$$
so $U_i$~is a compact subset of $\UU(\rational_p)$ that is normalized by both $K$ and~$C$. The maximality of~$C$ implies that $K_i \subseteq C \cdot \UU(\rational_p)$, so $K_i \subseteq C U_i$. Then, since every compactly generated subgroup of $\UU(\rational_p)$ is bounded \cite[Prop.~2.6.3, p.~46]{Abels-FinPres}, we conclude that the closure of
	$K C \cdot \langle U_1,\ldots,U_n \rangle$ is a compact subgroup of~$\GG(\rational_p)$ that contains $\langle K, K_1,\ldots,K_n \rangle$.
\end{proof}

We now prove the main theorem:

\begin{proof}[\bf Proof of \cref{GenSuper}]
By passing to a (torsion-free) subgroup of finite index in~$\Gamma$, we may assume $T_\Gamma \cap \MM_S K_T \UU_S = \{e\}$ (and also that $\Gamma = M_\Gamma T_\Gamma U_\Gamma$).
Let us assume $\Char L = 0$. (See \cref{charp} for the case where $\Char L \neq 0$.) Then, since $L$ is nonarchimedean, we may assume $L = \rational_p$, for some prime number~$p$, so we are given a homomorphism $\alpha \colon \Gamma \to \GL_n(\rational_p)$. We consider two cases.

\setcounter{case}{0}

\begin{case} \label{GenSuperPf-S}
Assume $p \in S$.
\end{case}
Write $\Zar{\alpha(\Gamma)} = \NN \times \CC \times \HH$ (after modding out a finite subgroup), where $\NN$ is semisimple, $\CC$ is a torus, and every reductive subgroup of~$\HH$ acts faithfully by conjugation on~$\unip \HH$. Then $\alpha$~can be decomposed into a homomorphism~$\alpha_N$ into~$\NN$,  a homomorphism~$\alpha_C$ into~$\CC$, and a homomorphism~$\alpha_H$ into~$\HH$. We consider these three components of~$\alpha$ separately.

Since $\Zar{\alpha_N( T_\Gamma U_\Gamma )}$ is a solvable, normal subgroup of the semisimple group~$\NN$, we know that it is finite. By modding it out, we may assume $\alpha_N$ is trivial on $T_\Gamma U_\Gamma$. Assumption~\fullref{GenSuper}{Msuper} provides a continuous homomorphism $\widehat\alpha_N \colon \MM_S \to \NN(\rational_p)$ whose restriction to~$M_\Gamma$ agrees with $\alpha_N |_{M_\Gamma}$ up to a bounded error. By modding out a finite subgroup of $\Zar{\alpha(\Gamma)}$, we may assume $\widehat\alpha_N$ is trivial on $\MM_S \cap \TT\!_S$. Then $\widehat\alpha_N$ can be extended to a continuous homomorphism $\widetilde\alpha_N \colon \GG_S \to \NN(\rational_p)$, by specifying that the extension is trivial on $\TT\!_S  \UU_S$. 
Then $\widetilde\alpha_N$ agrees with $\alpha_N$ up to a bounded error, so this deals with the part of~$\alpha$ that maps into~$\NN$. We may therefore assume, henceforth, that $\NN$ is trivial.

We can extend $\alpha_{C}|_{T_\Gamma}$ to a continuous homomorphism $\widehat\alpha_{C} \colon T_\Gamma \MM_S K_T \UU_S \to \CC(\rational_p)$, by specifying that the extension is trivial on the open, normal subgroup $\MM_S K_T \UU_S$. Now, let $K_C$ be the maximal compact subgroup of $\CC(\rational_p)$, so $\CC(\rational_p)/K_C$ is a finitely generated, torsion-free, abelian group. Then Assumption~\fullref{GenSuper}{UnoZ} implies that the image of $\alpha_C(U_\Gamma)$ in $\CC(\rational_p)/K_C$ must be trivial, which means $\alpha_C(U_\Gamma) \subseteq K_C$. Also, since $\CC$ is abelian, Assumption~\fullref{GenSuper}{Msuper} implies that $\alpha_C(M_\Gamma)$ is trivial (after passing to a finite-index subgroup). Then, for all $m \in M_\Gamma$, $t \in T_\Gamma$, and $u \in U_\Gamma$, we have
	$$\alpha_C(mtu) 
	= \alpha_C(m) \cdot  \alpha_C(t) \cdot \alpha_C(u) 
	\in e \cdot \widehat\alpha_C(mtu) \cdot K_C
	= \widehat\alpha_{C}(mtu) \,  K_C ,$$
so this deals with the part of~$\alpha$ that maps into~$\CC$. We may therefore assume, henceforth, that $\CC$ is trivial. 

We are now assuming that $\NN$ and~$\CC$ are trivial, so $\Zar{\alpha(\Gamma)} =  \HH$, which means that 
	\begin{align} \label{nocentraltori}
	\text{every reductive subgroup of $\Zar{\alpha(\Gamma)}$ acts faithfully on $\unip \Zar{\alpha(\Gamma)}$.}
	\end{align}

We know that $\Zar{\alpha(U_\Gamma)}$ is nilpotent (since $U_\Gamma$ is nilpotent), so it has a unique maximal torus~$\RR$. We also know that $\Zar{\alpha(U_\Gamma)}$ is a normal subgroup of~$\Zar{\alpha(\Gamma)}$
(since $U_\Gamma$ is a normal subgroup of~$\Gamma$). Therefore $\RR$ is a normal subgroup of $\Zar{\alpha(\Gamma)}$. Since any normal torus in a connected algebraic group is central, we conclude that $\RR$ is contained in the center of $\Zar{\alpha(\Gamma)}$. Then \pref{nocentraltori} implies that $\RR$ is trivial. This means that $\Zar{\alpha(U_\Gamma)}$ is unipotent.
Also, since $\UU$ is a unipotent $\rational$-group, and Assumption~\fullref{GenSuper}{GZ} tells us that $\UU(\integer) \almsubset U_\Gamma \almsubset \UU(\integer_S)$, we know that $U_\Gamma$ is Zariski-dense in~$\UU$ and $\rank U_\Gamma = \dim \UU$.
(To establish the equality, note that if $\Gamma^+$ is any finitely generated subgroup of $\UU(\rational)$ that contains $\UU(\integer)$, then the proof of \cite[Thm.~2.10, p.~32]{RaghunathanBook} shows that $\Gamma^+$ is a lattice in $\UU(\real)$, so \cite[Thm.~2.10, p.~32]{RaghunathanBook} implies $\rank\Gamma^+ = \dim \UU$.) Therefore, \cref{SuperRank} tells us that $\alpha|_{U_\Gamma}$ extends to a (unique) rational homomorphism $\alpha_U \colon \UU \to  \Zar{\alpha(U_\Gamma)}$. (Furthermore,  $\alpha_U$~is defined over~$\rational_p$.)

Let:
\noprelistbreak
	\begin{itemize}
	\item  $\graph(\alpha|_{M_\Gamma T_\Gamma}) = \bigl\{\, \bigl( \gamma, \alpha(\gamma) \bigr) \mid \gamma \in M_\Gamma T_\Gamma \,\bigr\}
	\subseteq M_\Gamma T_\Gamma\times \alpha(M_\Gamma T_\Gamma) 
	\subseteq  \bigl( \MM(\rational) \TT(\rational)  \bigr)  \times \GL_n(\rational_p)$, 
	\item $\Zar{\graph(\alpha|_{M_\Gamma T_\Gamma})}$ be the Zariski closure of $\graph(\alpha|_{M_\Gamma T_\Gamma})$ in $\MM\TT \times \Zar{\alpha(M_\Gamma T_\Gamma)}$,
	\item $\BB$ be a (reductive) Levi subgroup of $\Zar{\graph(\alpha|_{M_\Gamma T_\Gamma})}$ that is defined over~$\rational_p$,
	and
	\item $ \VV = \bigset{ v \in \Zar{\alpha(M_\Gamma T_\Gamma)} }{ (e , v) \in \Zar{\graph(\alpha|_{M_\Gamma T_\Gamma})} }$.
	\end{itemize}
By passing to a finite-index subgroup of~$\Gamma$, we may assume $\Zar{\graph(\alpha|_{M_\Gamma T_\Gamma})}$ is connected.

\begin{claim}[$\{e\} \times \VV$ is the unipotent radical of\/ $\Zar{\graph(\alpha|_{M_\Gamma T_\Gamma})}$, and $\VV$ is in the center of\/ $\Zar{\alpha(\Gamma)}$, so
	$$ \Zar{\graph(\alpha|_{M_\Gamma T_\Gamma})} = \BB \cdot \bigl( \{e\} \times \VV \bigr) \iso \BB \times \VV .$$]%
To verify this, first note that, since the reductive group $\MM\TT$ has no nontrivial normal unipotent subgroups, the unipotent radical of $\Zar{\graph(\alpha|_{M_\Gamma T_\Gamma})}$ must be contained in the kernel of the natural projection from $\Zar{\graph(\alpha|_{M_\Gamma T_\Gamma})}$ to~$\MM\TT$. This kernel is $\{e\} \times \VV$. 
We now prove the reverse inclusion. Since $M_\Gamma T_\Gamma$ normalizes $U_\Gamma$, we know that $\graph(\alpha|_{M_\Gamma T_\Gamma})$ normalizes $\graph(\alpha|_{U_\Gamma})$. Then the uniqueness of~$\alpha_U$ implies that $\graph(\alpha|_{M_\Gamma T_\Gamma})$ also normalizes $\graph(\alpha_U)$. Since $\graph(\alpha_U)$ is Zariski closed (because the homomorphism $\alpha_U$ is rational), we conclude that
$\Zar{\graph(\alpha|_{M_\Gamma T_\Gamma})}$ normalizes $\graph(\alpha_U)$.
So $\{e\} \times \VV$ normalizes $\graph(\alpha_U)$.
Hence, for any $v \in \VV$ and $u \in U_\Gamma$, we have
	$$ \bigl( u, v^{-1} \, \alpha_U(u) \, v\bigr)
	= (e,v)^{-1} \bigl( u, \alpha_U(u) \bigr) (e,v) 
	\in \graph(\alpha_U) ,$$
so $v^{-1} \, \alpha_U(u) \, v = \alpha_U(u)$.
This implies that 
	$$ \text{$\VV$ centralizes $\Zar{\alpha(U_\Gamma)}$} .$$
Also, since $T_\Gamma$ is central in $M_\Gamma T_\Gamma$, and Assumption~\fullref{GenSuper}{Msuper} tells us that $\Zar{\alpha(M_\Gamma)}$ has no unipotent radical, we know that $\unip \Zar{\alpha(M_\Gamma T_\Gamma)}$ is central in $\Zar{\alpha(M_\Gamma T_\Gamma)}$. Therefore, $\VV$ centralizes 
	$$\Zar{\alpha(U_\Gamma)} \cdot \unip \Zar{\alpha(M_\Gamma T_\Gamma)} = \unip \Zar{\alpha(\Gamma)},$$
so \pref{nocentraltori} tells us that $\VV$ is unipotent. Since $\VV$ is normal, this completes the proof that $\{e\} \times \VV$ is the unipotent radical of $\Zar{\graph(\alpha|_{M_\Gamma T_\Gamma})}$. In addition, $\VV$~is central in $\Zar{\alpha(\Gamma)}$, because it centralizes both $\Zar{\alpha(U_\Gamma)}$ and $\Zar{\alpha(M_\Gamma T_\Gamma)}$.  (The latter is because $\VV \subseteq \unip \Zar{\alpha(M_\Gamma T_\Gamma)}$.) This completes the proof of the claim.
\end{claim}%

Since $\BB \cap \bigl( \{e\} \times \VV \bigr)$ is trivial, the projection from $\BB$ onto~$\MM\TT$ has trivial kernel, so it is an isomorphism of algebraic groups. Therefore, $\BB$ is the graph of a rational homomorphism $\widehat\alpha_{MT} \colon \MM\TT \to \Zar{\alpha(M_\Gamma T_\Gamma)}$ that is defined over~$\rational_p$.

Define $\alpha_V \colon M_\Gamma T_\Gamma \to \VV(\rational_p)$ by $\widehat\alpha_{MT}(g) = \alpha(g) \, \alpha_V(g)$. (This is a homomorphism, since $\VV$ is central.)
Since $\VV$ is abelian, Assumption~\fullref{GenSuper}{Msuper} tells us that $\alpha_V(M_\Gamma)$ is trivial (after passing to a finite-index subgroup). Also, since $T_\Gamma$ is finitely generated \cite[Thm.~5.12, p.~176]{PlatonovRapinchukBook}, and $\VV$ is unipotent, we know that $\alpha_V(T_\Gamma)$ is contained in a compact subgroup~$K_V$ of~$\VV(\rational_p)$ \cite[Prop.~2.6.3, p.~46]{Abels-FinPres}.
Now, since $\BB \subseteq \Zar{\graph(\alpha|_{M_\Gamma T_\Gamma})}$ normalizes $\graph(\alpha_U)$, 
the product $\BB \cdot \graph(\alpha_U)$ is a subgroup of~$\GG \times \Zar{\alpha(\Gamma)}$. It is the graph of a rational homomorphism $\widehat\alpha \colon \GG \to \Zar{\alpha(\Gamma)}$ that is defined over~$\rational_p$, and satisfies
	$$ \text{$\widehat\alpha( gu ) = \widehat\alpha_{MT}(g) \, \alpha_U(u)$
	\ for $g \in \MM\TT$ and $u \in \UU$.} $$
Then, for $g \in M_\Gamma T_\Gamma$ and $u \in U_\Gamma$, we have
	$$ \widehat\alpha( gu ) 
	= \widehat\alpha_{MT}(g) \cdot \alpha_U(u)
	=  \alpha(g) \, \alpha_V(g) \cdot  \alpha(u)
	\in \alpha(g) K_V \cdot \alpha(u)
	= \alpha(gu) \cdot K_V .$$
Since $K_V \subset \VV$ is central in $\Zar{\alpha(\Gamma)}$, this completes the proof of this \lcnamecref{GenSuperPf-S}.

\begin{case} \label{GenSuperPf-notS}
Assume $p \notin S$.
\end{case}
Write $\Zar{\alpha(T_\Gamma)} = \RR \times \VV$, where $\RR$ is a torus and $\VV$ is unipotent. Since $\rational_p^\times = \langle p \rangle \times \text{compact}$, we may write $\RR(\rational_p) = Z \times E$, where $Z$ is free abelian, $E$~is compact, and every eigenvalue of every element of~$Z$ is a power of~$p$. Let $\alpha_Z$ be the projection of~$\alpha|_{T_\Gamma}$ to~$Z$, and extend it to a continuous homomorphism $\widehat\alpha \colon T_\Gamma   \MM_S K_T \UU_S \to Z$, by defining $\widehat\alpha$ to be trivial on the open, normal subgroup $\MM_S K_T \UU_S$.

Since $M_\Gamma \subseteq \MM(\integer_S)$ (up to finite index) and $p \notin S$, we know that $M_\Gamma$ is contained in the compact subgroup $\MM(\integer_p)$ of~$\MM(\rational_p)$. Hence, Assumption~\fullref{GenSuper}{Msuper} implies that $\alpha(M_\Gamma)$ is contained in a compact subgroup~$K_M$ of $\GL_n(\rational_p)$. Of course, we may assume $K_M \subseteq \Zar{\alpha(M_\Gamma)}$.
Also, since $T_\Gamma$ is finitely generated \cite[Thm.~5.12, p.~176]{PlatonovRapinchukBook}, the $p$-adic closure of the projection of~$\alpha(T_\Gamma)$ to~$\VV(\rational_p)$ is a compact subgroup $K_V$ \cite[Prop.~2.6.3, p.~46]{Abels-FinPres}. 

Furthermore, we now show that the $p$-adic closure of $\alpha(U_\Gamma)$ is a  compact subgroup~$K_U$. 
Write $\Zar{\alpha(U_\Gamma)} = \CC \times \WW$, where $\CC$~is a torus and $\WW$~is unipotent, and let $\alpha_C$ and~$\alpha_W$ be the projections of $\alpha|_{U_\Gamma}$ to the two direct factors. Just as in \cref{GenSuperPf-S}, we see that $\alpha_C(U_\Gamma)$ is contained in a compact subgroup~$K_C$ of $\CC(\rational_p)$.  Now, let $K$ be the $p$-adic closure of $\alpha_W \bigl(\UU(\integer) \cap \Gamma\bigr)$ in $\WW(\rational_p)$. Since $\UU(\integer)$ is finitely generated, we know that $K$ is compact \cite[Prop.~2.6.3, p.~46]{Abels-FinPres}. Let
	$$ \text{$K_W$ be the $p$-adic closure of 
	$\{\, \root m \of w \mid w \in K, \ m \in \integer^+, \ p \nmid m \,\} 
	\subseteq \WW(\rational_p)$} .$$
Then $K_W$ is also a compact subgroup of $\WW(\rational_p)$. (It is a subgroup by \cite[Prop.~2.4.3(a), p.~34]{Abels-FinPres}. It is compact because dividing an element of the Lie algebra of $\WW(\rational_p)$ by a scalar that is coprime to~$p$ does not change the $p$-adic norm of the vector.) After passing to a finite-index subgroup, Assumption~\fullref{GenSuper}{GZ} tells us that $U_\Gamma \subseteq \WW(\integer_S)$. Then, since $p \notin S$, we have $\alpha_W(U_\Gamma) \subseteq K_W$,
so $\alpha(U_\Gamma) \subseteq K_C K_W$ is contained in a compact subgroup, as desired.

For $m \in M_\Gamma$, $t \in T_\Gamma$ and $u \in U_\Gamma$, we have
	\begin{align*}
	 \widehat\alpha(mtu) 
		&= \alpha_Z(t)
		\in \alpha(t) \cdot E K_V
		= \alpha(m)^{-1} \cdot \alpha(m) \alpha(t) \alpha(u) \cdot \alpha(u^{-1}) E K_V
		\\&\subseteq K_M \cdot \alpha(mtu) \cdot K_U E K_V
		= \alpha(mtu) \cdot K_M K_U E K_V
		. \end{align*}
\cref{CpctSubsGenCpct} tells us that the closure of $\langle K_M K_U E K_V \rangle$ is compact.

To complete the proof, all that remains is to show that $\langle K_M K_U E K_V \rangle$ is centralized by $\widehat\alpha(\GG_S^\Gamma)$.
Since $\TT$ normalizes~$\UU$, we know that $\alpha(T_\Gamma)$ normalizes~$K_U$. It also normalizes (in fact, centralizes) $K_M$, $E$, and~$K_V$, since all three groups are contained in $\Zar{\alpha(M_\Gamma T_\Gamma)}$, whose center contains $\Zar{\alpha(T_\Gamma)}$. Therefore, $\alpha(T_\Gamma)$ normalizes $\langle K_M K_U E K_V \rangle$, which is also normalized by $E K_V$. Since $\alpha_Z(T_\Gamma) \subseteq \alpha(T_\Gamma) E K_V$, we conclude that $\alpha_Z(T_\Gamma)$ normalizes the closure of $\langle K_M K_U E K_V \rangle$, which is compact (as was already noted at the end of the preceding paragraph). However, any element~$z$ of~$Z$ is diagonalizable (since it is in a torus) and all of its eigenvalues are in~$\rational_p$ (indeed, they are powers of~$p$), so $z$~is diagonalizable over~$\rational_p$. Since all of its eigenvalues are are powers of~$p$, this implies that $z$~must centralize any compact subgroup of $\GL_n(\rational_p)$ that it normalizes. We conclude that $\widehat\alpha(\GG_S^\Gamma) = \alpha_Z(T_\Gamma)$ centralizes $\langle K_M K_U E K_V \rangle$, as desired.
\end{proof}

\begin{rem} \label{charp}
To complete the proof of \cref{GenSuper}, we treat the case where the characteristic of~$L$ is nonzero, by adapting \cref{GenSuperPf-notS} of the above proof.
Since $T_\Gamma$ and $U_\Gamma$ have finite Hirsch rank, and unipotent $L$-groups are torsion,  we may assume, after passing to a subgroup of finite index, that $\Zar{\alpha(T_\Gamma U_\Gamma)}$ is a torus~$\RR$.
Letting $p$ be a uniformizer of~$L$, we have $L^\times = \langle p \rangle \times \text{compact}$, so we may write $\RR(L) = Z \times E$, where $Z$ is free abelian, $E$~is compact, and every eigenvalue of every element of~$Z$ is a power of~$p$. Let $\alpha_Z$ be the projection of~$\alpha|_{T_\Gamma}$ to~$Z$, and extend it to a continuous homomorphism $\widehat\alpha \colon T_\Gamma   \MM_S K_T \UU_S \to Z$, by defining $\widehat\alpha$ to be trivial on the open, normal subgroup $\MM_S K_T \UU_S$.
From Assumption~\fullref{GenSuper}{UnoZ}, we see that $\alpha(U_\Gamma) \subseteq E$.
Also, since $\Char L \neq 0$, Assumption~\fullref{GenSuper}{Msuper} tells us that $\alpha(M_\Gamma)$ is contained in a compact subgroup~$K_M$ of $\GL_n(L)$. Then, for $m \in M_\Gamma$, $t \in T_\Gamma$ and $u \in U_\Gamma$, we have
	\begin{align*}
	 \widehat\alpha(mtu) 
		&= \alpha_Z(t)
		\in \alpha(t) \cdot E
		= \alpha(m)^{-1} \cdot \alpha(m) \alpha(t) \alpha(u) \cdot \alpha(u^{-1}) E
		\\&\subseteq K_M \cdot \alpha(mtu) \cdot E
		= \alpha(mtu) \cdot K_M E
		. \end{align*}
Furthermore, we may assume $K_M$ is contained in $\Zar{\alpha(M_\Gamma)}$, which centralizes $\RR(L) = Z \times E$ (assuming, as we may, that $\Zar{\alpha(M_\Gamma)}$ is connected), so $K_M E$ is a compact subgroup that centralizes $\widehat\alpha(\Gamma)$ (since $\widehat\alpha(\Gamma) \subseteq Z$).
\end{rem}

We will also use the following refinement of \cref{GenSuper}:

\begin{cor} \label{GenSuperFactor}
Assume the notation and hypotheses of \cref{GenSuper}. For each $p \in S$, choose\/ $\rational_p$-subgroups\/ $\AA_p$ and\/~$\BB_p$ of\/~$\GG$, such that 
\noprelistbreak
	\begin{enumerate} \renewcommand{\theenumi}{\roman{enumi}}
	\item $\GG = \AA_p \times \BB_p$ \textup(up to finite index\textup),
	and
	\item \label{GenSuperFactor-cpct}
	the projection of\/ $\Gamma$ to\/ $\BB_p$ is contained in a compact subgroup~$K_p$ of\/ $\BB_p(\rational_p)$.
	\end{enumerate}
Let 
\noprelistbreak
	\begin{itemize}
	\item $\AA_S = \GG(\real) \times \bigtimes_{p \in S} \AA_p(\rational_p)$,
	and
	\item $\AA_S^\Gamma$ be the image of\/ $\GG_S^\Gamma$ under the natural projection $\pi\!_A \colon \GG_S \to \AA_S$ with kernel\/ $\BB_S^f = \bigtimes_{p \in S} \BB_p(\rational_p)$.
	\end{itemize}
Then\/ $\Gamma$ \textup(or, more precisely, $\pi\!_A(\Gamma)$\textup) is $L$-superrigid in\/~$\AA_S^\Gamma$.
\end{cor}

\begin{proof}
Suppose $\alpha \colon \pi\!_A(\Gamma) \to \GL_n(L)$. Composing with $\pi\!_A$ yields a homomorphism $\alpha' \colon \Gamma \to \GL_n(L)$.
The proof of \cref{GenSuper} constructs only two kinds of extensions of~$\alpha'$. Namely, if we ignore a bounded error (and ignore passing to finite-index subgroups), and if, in \cref{GenSuperPf-S}, we consider only a single component $\alpha'_N$, $\alpha'_C$, or~$\alpha'_H$, then either:
	\begin{enumerate} \renewcommand{\theenumi}{\alph{enumi}}
	\item \label{GenSuperFactorPf-extend}
	$\alpha'$ extends to a continuous homomorphism $\widehat\alpha \colon \GG_S \to \GL_n(L)$,
	or
	\item \label{GenSuperFactorPf-T}
	$\alpha'$ factors through the projection $\pi_T \colon \GG_S^\Gamma \to T_\Gamma$.
	\end{enumerate}

In situation~\pref{GenSuperFactorPf-extend}, let 
	\begin{itemize}
	\item $\widetilde\alpha$ be the restriction of~$\widehat\alpha$ to $\AA_S$,
	and
	\item $K$ be the closure of the projection of~$\Gamma$ to~$\BB_S^f$.
	\end{itemize}
For $\gamma \in \Gamma$ we have (up to bounded error):
	$$ \alpha\bigl( \pi\!_A(\gamma) \bigr) 
	= \alpha'(\gamma)
	= \widehat\alpha(\gamma) 
	\in \widehat\alpha\bigl( \pi\!_A(\gamma) \bigr)  \cdot \widehat\alpha(K)
	= \widetilde\alpha\bigl( \pi\!_A(\gamma) \bigr) \cdot \widehat\alpha(K) .$$
From Assumption~\fullref{GenSuperFactor}{cpct}, we know that $K$ is a compact subgroup of~$\BB_S^f$. Since $\widehat\alpha$ is continuous, this implies that $\widehat\alpha(K)$ is compact. So $\alpha$ agrees with $\widetilde\alpha$ up to a bounded error on $\pi\!_A(\Gamma)$.

We now consider situation~\pref{GenSuperFactorPf-T}. Write 
	$\AA_S = M\!_A \, T\!_A \, U\!_A$,
where $M\!_A$~ is semisimple, $T\!_A$~is a torus, and $U\!_A$ is the unipotent radical. Let $C\!_A/(M\!_A U\!_A)$ be the (unique) maximal compact subgroup of the abelian group $\AA_S/(M\!_A U\!_A)$. Then we have 
	$$\AA_S^\Gamma = \pi\!_A(\GG_S^\Gamma) = \pi\!_A(T_\Gamma) \cdot C\!_A .$$
Assume $T_\Gamma$ is torsion free (by passing to a subgroup of finite index). Then, since Assumption~\fullref{GenSuperFactor}{cpct} implies that $\pi\!_A(T_\Gamma)$ is discrete in~$T\!_A$, we see that $\pi\!_A(T_\Gamma) \cap C\!_A = \{e\}$. 
Also note that the restriction of $\pi\!_A$ to~$\Gamma$ is bijective (since $\Gamma$ embeds in $\GG(\real)$, which is one of the factors in the definition of~$\AA_S$). Since we are in situation~\pref{GenSuperFactorPf-T}, this implies that $\alpha$~must factor through the projection $\AA_S^\Gamma \to \pi\!_A(T_\Gamma)$.
Therefore, $\alpha$~can be extended to a continuous homomorphism defined on $\AA_S^\Gamma$, by specifying that the extension is trivial on the open, normal subgroup~$C\!_A$. 

Finally, we remark that $\pi\!_A(\Gamma)$ can be identified with~$\Gamma$, since the restriction of $\pi\!_A$ to~$\Gamma$ is bijective (as was noted above).
\end{proof}

\section{Solvable groups} \label{SolvSect}

In this section, we use \cref{GenSuper} to establish several results on the superrigidity of $S$-arithmetic subgroups of solvable groups, after we recall the following useful observation.

\begin{lem}[{}{\cite[Lem.~7.5.4, p.~164]{Abels-FinPres}}] \label{noZ}
Let 
\noprelistbreak
	\begin{itemize}
	\item $F$ be an algebraic number field,
	\item $\UU$ be a unipotent algebraic group defined over~$F$,
	\item $S$ be a finite set of places of~$F$,
	and
	\item $\Gamma$ be an $S$-arithmetic subgroup of\/~$\UU$.
	\end{itemize}
If $S$ contains a nonarchimedean place, then\/ $\Gamma$ has no infinite, cyclic quotients.
\end{lem}

\begin{proof}[\bf Proof of \cref{SuperRes}]
By passing to a finite-index subgroup, we may assume $\GG$ is connected.
We may also assume $L$ is nonarchimedean, since \cite[Thm.~1.10]{Witte-NonarchSuper} treats the case where $L$ is archimedean. 
Therefore, it suffices to verify the hypotheses of \cref{GenSuper} with $\Zar{\Res{F/\rational} \Gamma}$ in the role of~$\GG$, and $\Char S$ in the role of~$S$.
Since $\Gamma$ is $S$-arithmetic, Assumptions \fullref{GenSuper}{GZ} and~\fullref{GenSuper}{MTU} are immediate.
The semisimple group~$\MM$ is trivial, since $\GG$ is solvable, so Assumption~\fullref{GenSuper}{Msuper} is trivially true. 

See \cref{noZ} for Assumption~\fullref{GenSuper}{UnoZ} if $S$ contains a nonarchimedean place. If not, then, by assumption, we must have $[\GG, \GG] = \unip \GG$, so $[\Gamma,\Gamma]$ is Zariski-dense in $\unip \GG$, and must therefore have finite index in $\Gamma \cap \UU$, so Assumption~\fullref{GenSuper}{UnoZ} is immediate.
\end{proof}

\begin{rem}
Note that if $\Gamma$ is superrigid in~$G$, then $\Gamma$~is also superrigid in $G \times H$, for any group~$H$. Therefore, in the conclusion of \cref{SuperRes}, the group $\Zar{\Res{F/\rational} \Gamma}$ can be replaced with $\Zar{\Res{F/\rational} \Gamma} \cdot \Res{F/\rational} \ZZ$, where $\ZZ$~is the center of~$\GG$.
\end{rem}

\Cref{SuperRes} is somewhat unsatisfactory, because the hypotheses deal with objects $\GG$ and~$S$ that are defined over the number field~$F$, but the conclusion replaces them with corresponding objects over~$\rational$. The following result eliminates this shortcoming, at the expense of a Zariski-density assumption.

\begin{cor} \label{SuperF}
Assume $F$, $\GG$, $S$, and~$\Gamma$ are as in \cref{SuperRes}.
If\/ $\Gamma$ is Zariski-dense in $\Res{F/\rational} \GG$ and either $S$~contains a nonarchimedean place or\/ $[\GG^\circ, \GG^\circ] = \unip \GG$, then\/ $\Gamma$ is $L$-superrigid in\/ $\GG_S^\Gamma$, for every local field~$L$.
\end{cor}

\begin{proof}
For each $p \in \Char S$, let 
	\begin{itemize}
	\item $S_p = \{\, v \in S \mid \Char v = p \,\}$,
	
	\item $\displaystyle{ \GG_p = \bigtimes_{v \in S_p} \GG(F_v) }$,
	
	\item $\overline{S_p} = \{\, \text{nonarchimedean places~$v$ of~$F$} \mid \text{$\Char v = p$} \,\}$,
	and 
	
	\item $\GG_{\widehat p} = \bigtimes_{v \in \overline{S_p} \smallsetminus S_p} \GG(F_v)$. 
	\end{itemize}
Since \cite[Thm.~1.10]{Witte-NonarchSuper} treats the case where $L$ is archimedean,
we may assume $L$ is nonarchimedean.
Then the hypotheses of \cref{GenSuper} can be verified as in the proof of \cref{SuperRes}, so \cref{GenSuperFactor} applies, because 
	$$\bigl( \Res{F/\rational}\GG \bigr) \bigl( \rational_p \bigr) 
	\iso \GG (F \otimes_{\rational} \rational_p)
	\iso \GG \Bigl( \bigoplus\nolimits_{v \in \overline{S_p}} F_v \Bigr)
	\iso \bigtimes_{v \in \overline{S_p}} \GG(F_v)
	\iso \GG_p \times \GG_{\widehat p}
	, $$
and the image of $\Gamma$ in $\GG_{\widehat p}$ is compact by definition.
\end{proof}

\begin{cor} \label{UnipSuper}
If\/ $\UU$ is a unipotent algebraic group over an algebraic number field~$F$, and $S$ is a finite set of places of~$F$ that contains a nonarchimedean place, then every $S$-arithmetic subgroup of\/~$\UU$ is $L$-superrigid in\/ $\UU_S$, for every local field~$L$.
\end{cor}

\begin{proof}
Every ($S$-)arithmetic subgroup of~$\UU$ is Zariski-dense in $\Res{F/\rational} \UU$. (Indeed, if $\mathcal{O}$ is the ring of integers of~$F$, then it is well known that $\UU(\mathcal{O}) \iso (\Res{F/\rational} \UU)(\integer)$ is a lattice in $(\Res{F/\rational} \UU)(\real)$, and is therefore Zariski-dense.) Hence, the desired conclusion is immediate from \cref{SuperF}.
\end{proof}

When $F = \rational$, we can say a bit more in this unipotent case.

\begin{defn}
 A (countable) subgroup~$\Gamma$ of a topological group~$G$ is \emph{\underline{strictly} $L$-superrigid} in~$G$ if $\Gamma$ is $L$-superrigid in~$G$ and the subgroup~$K$ in \cref{SuperDefn} can always be taken to be trivial.
\end{defn}

\begin{prop} \label{StrictlyQ}
Let 
\noprelistbreak
	\begin{itemize}
	\item $\UU$ be a unipotent algebraic group defined over\/~$\rational$,
	\item $S$ be a finite set of valuations of\/~$\rational$,
	\item $\Gamma$ be an $S$-arithmetic subgroup of\/~$\UU$,
	and
	\item $L$ be a local field.
	\end{itemize}
If $S$ contains at least one nonarchimedean place, then\/ $\Gamma$ is $L$-superrigid in $\UU_S$.

More precisely, suppose $\alpha \colon \Gamma \to \GL_n(L)$ is any homomorphism. Then:
\noprelistbreak
	\begin{enumerate}
	
	\item \label{StrictlyQ-cpct}
	If either $\Char L \neq 0$, or $L$~is a finite extension of~$\rational_p$, with $p \notin S$, then $\alpha(\Gamma)$ is contained in a compact subgroup of $\GL_n(L)$.
	
	\item \label{StrictlyQ-inS}
	If $L = \rational_p$, for some $p \in S$, then there exists a unique rational homomorphism $\widehat\alpha \colon \UU \to \GGL_n$, defined over\/~$\rational_p$, such that, for some finite-index subgroup\/~$\Gamma\!_0$ of\/~$\Gamma$, we have $\alpha(\gamma) = \widehat\alpha(\gamma)$ for all $\gamma \in \Gamma\!_0$. Therefore, $\Gamma$ is strictly $\rational_p$-superrigid.
	
	\end{enumerate}
\end{prop}

\begin{proof}
\pref{StrictlyQ-cpct}~This is implicit in \cref{charp} and \cref{GenSuperPf-notS} of the proof of \cref{GenSuper}.

\pref{StrictlyQ-inS}~Recall that a group~$G$ is said to be ``$p$-radicable'' if every element of~$G$ has a $p$th root in~$G$.
Since $p \in S$, it is obvious that the additive abelian group $\integer_S$ is $p$-radicable. Since $\UU/[\UU, \UU]$ is a direct sum of $1$-dimensional unipotent groups, we conclude that $\Gamma$ has a finite-index subgroup~$\Gamma'$ whose abelianization is $p$-radicable. By a straightforward induction on the nilpotence class of~$\Gamma'$, this implies that $\Gamma'$~is $p$-radicable \cite[Prop.~2.4.2, p.~33]{Abels-FinPres}.

By passing to a finite-index subgroup of~$\Gamma$, we may assume $\Zar{\alpha(\Gamma)}$ is connected. 
Then, since $\alpha(\Gamma)$, like~$\Gamma$, is nilpotent, we may write $\Zar{\alpha(\Gamma)} = \CC \times \VV$, where $\CC$ is a torus and $\VV$ is unipotent. Let $\alpha_C \colon \Gamma \to \CC$ be the projection of~$\alpha$ to~$\CC$.
As in \cref{GenSuperPf-S} of the proof of \cref{GenSuper}, we see that $\alpha_C(\Gamma)$ is contained in the maximal compact subgroup~$E$ of $\CC(\rational_p)$. 

However:
	\begin{itemize}
	\item  $E$ has a finite-index subgroup that is pro-$p$ (cf.\ \cite[Lem.~3.8, p.~138]{PlatonovRapinchukBook}),
	and 
	\item no $p$-radicable group has a nontrivial homomorphism to any pro-$p$ group. 
	\end{itemize}
Therefore, $\alpha_C(\Gamma)$ must be finite. Hence, by passing to a finite-index subgroup, we may assume that $\alpha_C(\Gamma)$ is trivial. Then $\CC$ is trivial, so $\Zar{\alpha(\Gamma)} = \VV$ is unipotent.
Hence, \cref{SuperRank} provides an extension of~$\alpha$ to a rational homomorphism $\widehat\alpha \colon \UU \to \GGL_n$.
\end{proof}

\begin{rem}
Some unipotent $S$-arithmetic groups are strictly superrigid, and others are not.
\begin{enumerate}

\item The assumption that $S$ contains a nonarchimedean place is necessary in \cref{StrictlyQ} (unless $\UU$ is trivial or $L$~is archimedean). To see this, suppose $S$ has no nonarchimedean places and $\UU$ is nontrivial. Then $\UU_S = \UU_\emptyset = \UU(\real)$ is connected, so there is no nontrivial, continuous homomorphism from $\UU_S$ to the totally disconnected group $\rational_p^\times = \GL_1(\rational_p)$. Therefore, if $\Gamma$ were $\rational_p$-superrigid in $\UU_S$, then every homomorphism from $\Gamma$ to~$\rational_p^\times$ would have bounded image.  However, to the contrary, $\Gamma \doteq \UU(\integer)$ is a nontrivial, finitely generated, torsion-free, nilpotent group, so there exists a nontrivial homomorphism from $\Gamma$ to~$\integer$. (That is, Assumption~\fullref{GenSuper}{UnoZ} fails.) Hence, there is a homomorphism $\alpha \colon \Gamma \to \rational_p^\times$ with unbounded image, for any prime~$p$.

\item In the special case of \cref{UnipSuper} in which $S$ contains every valuation~$v$ of~$F$ with $\Char v = p$, the group~$\Gamma$ is \emph{strictly} $\rational_p$-superrigid in $\UU_S$. 
 (In fact, $\Gamma$ is strictly $\rational_p$-superrigid in~$\UU_p$.)
To prove this, note that $1/p$ is an $S$-integer (in other words, $1/p \in \mathcal{O}_S$), so the conclusion follows from the argument in the first and third paragraphs of the proof of \fullcref{StrictlyQ}{inS} (with $\mathcal{O}_S$ in the place of~$\integer_S$).

\item On the other hand, there do exist cases of \cref{UnipSuper} in which the superrigidity of~$\Gamma$ is not strict. 
For example, let $F = \rational[i]$ and $p = 5$. We have $p = ab$ with $a = 2 + i$ and $b = 2 - i$. Let $v$ be the valuation corresponding to the prime ideal $(a)$, so $\Gamma = \integer[i, 1/a]$ is a $\{v\}$-arithmetic subgroup of the one-dimensional unipotent group~$F$. 
Since $p \Gamma = (1/a)p \Gamma = b \Gamma$, and the norm of~$b$ is~$p$, we see that $\Gamma\!/p\Gamma$ is cyclic of order~$p$, so $\integer + p \Gamma = \Gamma$. 
By induction on~$k$, this implies $\integer + p^k \Gamma = \Gamma$ for all $k \in \integer^+$, so $\Gamma \! / p^k \Gamma \iso \integer/p^k \integer$. Passing to the projective limit yields a surjective homomorphism~$\alpha$ from~$\Gamma$ onto $\integer_p$. The resulting composite homomorphism 
	$$ \Gamma \stackrel{\alpha}{\to} \integer_p \hookrightarrow \integer \times \integer_p \times \frac{\integer}{(p-1)\integer} \iso \rational_p^\times $$
does not extend to a continuous homomorphism defined on all of~$\rational_p$ (since $\rational_p$ is $p$-radicable).

\end{enumerate}
\end{rem}

Suppose $\GG$ is a solvable algebraic group over~$\rational$. It is well known that if $\Qrank \GG \neq 0$, then $S$-arithmetic subgroups of~$\GG$ are \emph{not} lattices in~$\GG_S$. Rather, they are lattices in a certain subgroup that is usually denoted~$\GG\one$ \cite[pp.~263--264]{PlatonovRapinchukBook}. In general, a finite-index subgroup of~$\GG_S^\Gamma$ is contained in~$\GG\one$, and may be much smaller, but the groups are commensurable if $\Rrank \GG = \Qrank \GG$. Therefore:

\begin{cor}
Suppose $\Gamma$~is a Zariski-dense, $S$-arithmetic subgroup of a solvable algebraic $\rational$-group\/~$\GG$, and $S$~contains a nonarchimedean place. Assume, by passing to a finite-index subgroup if necessary, that $\Gamma \subseteq \GG\one$. If $\Rrank \GG = \Qrank \GG$, then $\Gamma$~is $L$-superrigid in $\GG\one$, for every local field~$L$.
\end{cor}

\section{Groups that are not solvable} \label{NonsolvSect}

Let us recall the famous Margulis Superrigidity Theorem, which tells us that Assumption~\fullref{GenSuper}{Msuper} is often true.

\begin{defn}
Suppose $\GG$ is an algebraic group over an algebraic number field~$F$, and $S$~is a finite set of places of~$F$. The \emph{$S$-rank} of~$\GG$ is
	$ \sum_{v \in S \cup S_\infty} \rank_{F_v} \GG$,
where $S_\infty$ is the set of archimedean places of~$F$.
\end{defn}

\begin{thm}[{Margulis \cite[8.B, pp.~258--259]{MargBook}}]
Suppose
\noprelistbreak
	\begin{itemize}
	\item $\MM$ is a connected, semisimple algebraic group over a number field~$F$,
	\item $S$ is a finite set of places of~$F$,
	and
	\item the $S$-rank of every $F$-simple factor of\/~$\MM$ is at least two. 
	\end{itemize}
Then every $S$-arithmetic subgroup of\/~$\MM$ is semisimply superrigid in\/~$\MM_S$.
 \end{thm}
 
 Combining this with \cref{GenSuper} and \cref{GenSuperFactor} (and the arguments of \cref{SolvSect}) yields:

\begin{cor} \label{NonSolv}
 Let
 \noprelistbreak
 	\begin{itemize}
	\item $\GG$ be a connected algebraic group over a number field~$F$, 
	\item $S$ be a finite set of places of~$F$,
	\item $\Gamma$ be an $S$-arithmetic subgroup of~$\GG$ that is Zariski-dense,
	and
	\item $L$ be a local field.
	\end{itemize}
Assume:
	\begin{itemize}
	\item the $S$-rank of every $F$-simple factor of\/ $\GG/\Rad \GG$ is at least
two,
	and
	\item either $S$ contains a nonarchimedean place or $\unip \GG \subseteq [\GG,\GG]$.
	\end{itemize}
Then:
 \noprelistbreak
	\begin{enumerate}
	\item \label{NonSolv-res}
	$\Gamma$ is $L$-superrigid in\/ $\bigl(\Zar{\Res{F/\rational} \Gamma}\bigr){}_{\Char S}^\Gamma$.
	\item \label{NonSolv-GS}
	If\/ $\Gamma$ is Zariski-dense in $\Res{F/\rational} \GG$, then\/ $\Gamma$ is $L$-superrigid in\/ $\GG_S^\Gamma$.
	\end{enumerate}
 \end{cor}

\begin{ack}
  D.\,W.\,M.\ would like to thank the faculty, staff, and students in the mathematics department of the University of Chicago for their hospitality and financial support during visits while much of this research was carried out.
\end{ack}

\end{document}